\documentclass[11pt]{amsart}
\usepackage[margin=0.75in]{geometry}
\usepackage[T1]{fontenc}
\usepackage{amssymb,amsmath,mathtools,amsthm,thmtools,amsfonts}
\usepackage{dsfont} 
\usepackage{comment}
\usepackage{thm-restate}
\usepackage{caption}
\usepackage{url} 
\usepackage{float}
\usepackage{hyperref}
\hypersetup{
colorlinks,
linkcolor={red}, 
citecolor={black},
urlcolor={blue!60!black},
pdftitle={Erd\H{o}s meet Nash-Williams},
pdfauthor={Michelle Delcourt, Cicely Henderson, Thomas Lesgourgues, Luke Postle}}

\usepackage[utf8]{inputenc}
\usepackage{color}

\usepackage{enumitem}
\setlist[enumerate,1]{label={(\roman*)}}
\setlist[enumerate,2]{label={(\alph*)}}
\setlist[enumerate,3]{label={(\arabic*)}}

\setlist[itemize]{nolistsep,noitemsep, topsep=.1in}
\setlist[enumerate]{nolistsep,noitemsep, topsep=.1in}

\newcommand{\mc}[1]{\mathcal{#1}}%

\DeclarePairedDelimiter{\set}{\{}{\}}

\DeclarePairedDelimiter{\floor}{\lfloor}{\rfloor}

\usepackage[noabbrev,capitalise]{cleveref}

\theoremstyle{plain}
\newtheorem{thm}{Theorem}[section]
\newtheorem{lem}[thm]{Lemma}
\newtheorem{proposition}[thm]{Proposition}

\newtheorem{conj}[thm]{Conjecture}

\usepackage{etoolbox}
 \AtEndEnvironment{proof}{\setcounter{claim}{0}}

\theoremstyle{definition}
\newtheorem{definition}[thm]{Definition}

\newtheorem{remark}[thm]{Remark}

\crefname{equation}{}{}
\crefname{lem}{Lemma}{Lemmas}
\crefname{claim}{Claim}{Claims}
\crefname{thm}{Theorem}{Theorems}
\crefname{conj}{Conjecture}{Conjectures}
\crefname{enumi}{}{}

\usepackage[dvipsnames, table]{xcolor}

\title[Beyond Nash-Williams: Counterexamples to Clique Decomposition Thresholds]{Beyond Nash-Williams: Counterexamples to Clique Decomposition Thresholds for All Cliques Larger than Triangles}

\author{
Michelle Delcourt
 }
 \address{Department of Mathematics, Toronto Metropolitan University (formerly named Ryerson University), Toronto, Ontario M5B 2K3, Canada} \email{mdelcourt@torontomu.ca}
 \thanks{Delcourt's research supported by NSERC under Discovery Grant No. 2019-04269 and a Sloan Research Fellowship.}

\author{
 Cicely (Cece) Henderson
 }

 \author{
 Thomas Lesgourgues
 }
 \address{School of Mathematics and Statistics, UNSW Sydney, NSW 2052, Australia.} 
 \email{t.lesgourgues@unsw.edu.au}
 \thanks{This work was completed while the third author was at the University of Waterloo.}

 \author{
 Luke Postle
 }
 \address{Combinatorics and Optimization Department, University of Waterloo, Waterloo, Ontario N2L 3G1, Canada}
 \email{\{c3hender, lpostle\}@uwaterloo.ca}
 \thanks{Postle's research supported by NSERC under Discovery Grant No.\ 2019-04304.}
 \date{\today}

\usepackage{tikz}
\usepackage{calc}
\usetikzlibrary{calc,patterns,decorations.pathreplacing,backgrounds,shapes}

\newcommand{\TikzCaseFive}[1]{

\begin{tikzpicture}[x=1 cm, y=#1 cm, rounded corners]

    \draw[gray!40,fill=gray!20] (-1.75,1.3) rectangle (12,-1.6) ;
    \node at (9.5,-.5) [label=0:small edges]{};

    \draw[gray!80,fill=gray!40] (-1.5,1.25) rectangle (9.5,-1.5) ;
    \node at (3,-1.25) [label=0:cross edges]{};
    
    \begin{scope}
        \draw[gray,fill=gray!80] (-1,1) rectangle (2,-1) ;
        \node at (0,0.25) [label=0:$H$]{};
        \node at (-.75,-.25) [label=0:internal edges]{};        
    \end{scope}

    \begin{scope}[shift={(3.5,0)}]
        \draw[gray,fill=gray!80] (-1,1) rectangle (2,-1) ;
        \node at (0,0.25) [label=0:$H$]{};
        \node at (-.75,-.25) [label=0:internal edges]{};      
    \end{scope}

    \begin{scope}[shift={(7,0)}]
        \draw[gray,fill=gray!80] (-1,1) rectangle (2,-1) ;
        \node at (0,0.25) [label=0:$H$]{};
        \node at (-.75,-.25) [label=0:internal edges]{};       
    \end{scope}

    \begin{scope}[shift={(10.5,.5)}]
        \draw[gray,fill=gray!80] (-0.5,.5) rectangle (1,-.5) ;
        \node at (-0.1,0.25) [label=0:$I$]{};
        \node at (-0.1,-.25) [label=0:$\emptyset$]{};
    \end{scope}

\end{tikzpicture}}

\newcommand{\TikzCaseFour}[1]{

\begin{tikzpicture}[x=1 cm, y=#1 cm, rounded corners]

    \draw[gray!40,fill=gray!20] (-1.75,1.3) rectangle (10,-1.6) ;
    \node at (6.75,-.5) [label=0:small edges]{};

    \draw[gray!80,fill=gray!40] (-1.5,1.25) rectangle (6,-1.5) ;
    \node at (3,-1.25) [label=0:cross edges]{};
    
    \begin{scope}
        \draw[gray,fill=gray!80] (-1,1) rectangle (2,-1) ;
        \node at (0,0.25) [label=0:$H$]{};
        \node at (-1.12,-.25) [label=0:$H$-internal edges]{};        
    \end{scope}

    \begin{scope}[shift={(3.5,0)}]
        \draw[gray,fill=gray!80] (-1,1) rectangle (2,-1) ;
        \node at (0,0.25) [label=0:$H$]{};
        \node at (-1.12,-.25) [label=0:$H$-internal edges]{};      
    \end{scope}

    \begin{scope}[shift={(7,.5)}]
        \draw[gray,fill=gray!80] (-0.5,.5) rectangle (2.5,-.5) ;
        \node at (.5,0.25) [label=0:$I$]{};
        \node at (-0.6,-.25) [label=0:$I$-internal edges]{};
    \end{scope}

\end{tikzpicture}}
\begin{document}

\begin{abstract}
A central open question in extremal design theory is Nash-Williams' Conjecture from 1970 that every $K_3$-divisible graph on $n$ vertices (for $n$ large enough) with minimum degree at least $3n/4$ has a $K_3$-decomposition. A folklore generalization of Nash-Williams' Conjecture extends this to all $q\ge 4$ by positing that every $K_q$-divisible graph on $n$ vertices (for $n$ large enough) with minimum degree at least $\left(1-\frac{1}{q+1}\right)n$ has a $K_q$-decomposition. We disprove this conjecture for all $q\ge 4$; namely, we show that for each $q\ge 4$, there exists $c > 1$ such that there exist infinitely many $K_q$-divisible graphs $G$ with minimum degree at least $\left(1-\frac{1}{c\cdot(q+1)}\right)v(G)$ and no $K_q$-decomposition; indeed we construct them admitting no fractional $K_q$-decomposition thus disproving the fractional relaxation of this conjecture. Our result also disproves the more general partite version. Indeed, we even show the folklore conjecture is off by a multiplicative factor by showing that for every $\varepsilon > 0$ and every large enough integer $q$, there exist infinitely many $K_q$-divisible graphs $G$ with minimum degree at least $\bigg(1-\frac{1}{\left(\frac{1+\sqrt{2}}{2}-\varepsilon\right)\cdot (q+1)}\bigg)v(G)$ with no (fractional) $K_q$-decomposition.
\end{abstract}

\maketitle
 

\section{Introduction}

One of the most studied objects in design theory is a \emph{Steiner triple system} on $n$ elements, that is, a set $S$ of triples of an $n$-element set $X$ such that every pair of elements of $X$ is in exactly one triple in $S$. From a graph theoretic perspective, this is equivalent to a decomposition of the edges of the complete graph $K_n$ into edge-disjoint triangles; such a decomposition of the edges of a graph $G$ is called a \emph{triangle decomposition} or \emph{$K_3$-decomposition} of $G$. The obvious necessary conditions for a graph $G$ to admit a $K_3$-decomposition are that $3~|~e(G)$ and $2~|~d_G(v)$ for every $v\in V(G)$; such a graph is said to be \emph{$K_3$-divisible}. One of the oldest theorems in design theory is a result of Kirkman~\cite{K47} from 1847 that proves that the complete graph $K_n$ has a $K_3$-decomposition whenever $K_n$ is $K_3$-divisible (i.e.,~$n\equiv 1,3 \pmod 6$). 

A natural question is when a partial Steiner triple system (i.e.,~a set of edge-disjoint triangles) can be extended to a Steiner triple system. More generally, one may wonder whether `nearly' complete graphs admit a triangle decomposition provided they are $K_3$-divisible. Note one must impose some condition, such as high minimum degree, to ensure that every edge is in at least one triangle. The most central and famous conjecture in this direction is the following conjecture of Nash-Williams from 1970~\cite{nash1970unsolved}.

\begin{conj}[Nash-Williams~\cite{nash1970unsolved}]\label{conj:NW}
The following holds for sufficiently large $n$: If $G$ is a $K_3$-divisible graph on $n$ vertices with minimum degree at least $\frac{3}{4}n$,  then $G$ admits a $K_3$-decomposition.
\end{conj}

In an addendum to Nash-Williams' article~\cite{nash1970unsolved}, Graham provided a construction showing that the fraction $3/4$ is necessary. Indeed, his construction also provides graphs with minimum degree $\frac{3}{4}n-1$ and no \emph{fractional $K_3$-decomposition} (an assignment of non-negative weights to the triangles of a graph $G$ such that for every edge $e$, the total weight of triangles containing $e$ is exactly $1$). Since a $K_3$-decomposition is also a fractional $K_3$-decomposition, the existence of a fractional $K_3$-decomposition is a necessary condition for the existence of a $K_3$-decomposition. Graham's construction thus shows that the fraction $3/4$ would also be tight for the \emph{fractional relaxation} of Nash-Williams' Conjecture. 

Nash-Williams' Conjecture may be seen as the foundational problem in what may be termed extremal design theory; it is instructive to compare this value to the fractions in the minimum degree thresholds for other extremal problems: namely the threshold is $1/2$ for the existence of at least one triangle (which follows from a classical theorem usually attributed to Mantel~\cite{mantel} from 1907) and $2/3$ for the existence of a \emph{triangle factor}, that is a set of vertex-disjoint triangles spanning all vertices (as given by the celebrated Corr\'adi-Hajnal Theorem~\cite{CH63}). The current best-known results on Nash-Williams' Conjecture are as follows: In 2016, Barber, K\"{u}hn, Lo, and Osthus~\cite{BKLO16} proved that if the fractional relaxation of Nash-Williams' Conjecture holds for minimum degree $cn$ for some constant $c\ge 3/4$, then Nash-Williams' Conjecture holds for any constant $c' > c$; combined with the best-known fractional result of Delcourt and Postle~\cite{DP2021progress} from 2021, this confirms that Nash-Williams' Conjecture holds for graphs with minimum degree at least $0.82733n$. This improved a series of earlier works on the fractional relaxation by Garaschuk~\cite{garaschuk2014linear} in 2014, Dross~\cite{dross2016fractional} in 2016, and Dukes and Horsley~\cite{dukes_minimum_2020} in 2020. Nash-Williams' Conjecture has become a central open question in design theory as various generalizations are now tied to it (e.g., the minimum degree threshold for an $F$-decomposition for any $3$-chromatic $F$~\cite{GKLMO19}, the minimum degree threshold for approximately packing families of $2$-regular graphs~\cite{CKKO19}, and the minimum degree threshold for high-girth triangle decompositions~\cite{DHLP25}). For more history on Nash-Williams' Conjecture and its generalizations, see the survey by Glock, K\"uhn, and Osthus~\cite{GKO20Survey}.

A core area of design theory is to consider analogous questions for decompositions into cliques larger than triangles. To this end, an \emph{$(n,q,2)$-Steiner system} is a set $S$ of $q$-subsets of an $n$-element set $X$ such that every pair of elements of $X$ is in exactly one element of $S$. From a graph theoretic perspective, this is equivalent to a decomposition of the edges of the complete graph $K_n$ into edge-disjoint copies of $K_q$; such a decomposition of the edges of a graph $G$ is called a \emph{$K_q$-decomposition} of $G$. The obvious necessary conditions for a graph $G$ to admit a $K_q$-decomposition are that $\binom{q}{2}~|~e(G)$ and $(q-1)~|~d_G(v)$ for every $v\in V(G)$; such a graph is said to be \emph{$K_q$-divisible}. In the 1970s, Wilson~\cite{W72-EC1, W72-EC2, W75-EC3} proved for all $q\ge 3$ that $K_n$ admits a $K_q$-decomposition provided it is $K_q$-divisible and $n$ is sufficiently large. This result revolutionized design theory and is the `graph' case of the infamous Existence Conjecture of Combinatorial Designs from the mid-1800s (only recently settled in full by Keevash~\cite{K14} in 2014). 

The following is a folklore generalization of Nash-Williams' Conjecture for $K_q$-decompositions (dating at least to Gustavsson's 1991 thesis~\cite{gustavsson1991decompositions}), a conjectured high minimum degree version of Wilson's Theorem.

\begin{conj}\label{conj:Folklore}
For each integer $q\ge 4$, the following holds for sufficiently large $n$: If $G$ is a $K_q$-divisible graph on $n$ vertices with minimum degree $\delta (G) \geq \left(1-\frac{1}{q+1}\right)n$,  then $G$ admits a $K_q$-decomposition.
\end{conj}

The above conjecture coincides with Nash-Williams' Conjecture when $q=3$. In 1991, Gustavsson~\cite{gustavsson1991decompositions} in his thesis provided a construction to show that this minimum degree value would be tight (this construction was reiterated by Yuster~\cite{Yuster2005asymptotically} in 2005). Indeed, his construction admits no fractional $K_q$-decomposition and so this value would also be tight for the fractional relaxation of the above conjecture. Again, it is instructive to compare this value to the fractions in the minimum degree thresholds for other extremal problems: namely a value of $1-\frac{1}{q-1}$ for the existence of at least one copy of $K_q$ which follows from the classical Tur\'an's Theorem~\cite{turan1941extremal} from 1941, and $1-\frac{1}{q}$ for the existence of a \emph{$K_q$-factor}, that is a set of vertex-disjoint $K_q$'s spanning all vertices, which follows from the famous Hajnal-Szemer\'edi Theorem~\cite{HS70} from 1970.  

Beyond the pleasing symmetry of Conjecture~\ref{conj:Folklore} with the corresponding extremal results noted above, there was even more evidence in favor of the conjecture in the form of the best-known results. In 2019, Glock, K\"{u}hn, Lo, Montgomery, and Osthus~\cite{GKLMO19} proved that if the fractional relaxation of Conjecture~\ref{conj:Folklore} holds for minimum degree $cn$ for some constant $c\ge 1-\frac{1}{q+1}$, then Conjecture~\ref{conj:Folklore} holds for any constant $c' > c$.  Meanwhile, Montgomery~\cite{montgomery_fractional_2019} in 2019 produced the best-known result on the fractional relaxation by obtaining a minimum degree bound of $(1-\frac{1}{100q})n$, notably the first bound whose denominator is linear in $q$. This improved upon a series of earlier works on the fractional relaxation by Yuster~\cite{Yuster2005asymptotically} in 2005, Dukes~\cite{D12, D15} in 2012, and Barber, K\"{u}hn, Lo, Montgomery, and Osthus~\cite{BKLMO17} in 2017. Conjecture~\ref{conj:Folklore} has also become paramount in the realm of extremal design theory as more variations are tied to it (such as the minimum degree threshold for an $F$-decomposition for any $q$-chromatic $F$~\cite{GKLMO19} and the minimum degree threshold for approximately packing certain families of $(q-1)$-regular graphs~\cite{CKKO19}).  

In light of this history, our main result may then come as a bit of surprise: Conjecture~\ref{conj:Folklore} is false for all $q\ge 4$ as follows.

\begin{thm}\label{thm:mainExistence}
     For each integer $q\geq 4$, there exist a $c>1$ and infinitely many $K_q$-divisible graphs $G$ with minimum degree at least $\big(1-\frac{1}{c\cdot (q+1)}\big)v(G)$ with no $K_q$-decomposition and even no fractional $K_q$-decomposition.
\end{thm}

Indeed, the folklore conjecture is not only false but the multiplicative factor is wrong as follows (namely any factor $c$ arbitrarily close to $\frac{1+\sqrt{2}}{2}\approx 1.207$ still yields a counterexample provided $q$ is large enough).

\begin{thm}\label{thm:mainAsymptotic}
    Let $c=\frac{1+\sqrt{2}}{2}$ and fix $\varepsilon>0$. For every large enough integer $q$, there exist infinitely many $K_q$-divisible graphs $G$ with minimum degree at least $\big(1-\frac{1}{(c-\varepsilon)\cdot (q+1)}\big)v(G)$ with no $K_q$-decomposition and even no fractional $K_q$-decomposition.
\end{thm}

We note that in light of our results above, this may be the first instance in the literature of a problem where the ``threshold to embed absorbers'', namely $1-\frac{1}{q+1}$ via the result of Glock, K\"uhn, Lo, Montgomery, and Osthus~\cite{GKLMO19}, is below the ``fractional threshold'' (there are instances, such as in hypergraph matching, where the reverse is true; e.g.,~\cite{RRS06, RRS09}). Altogether, this makes determining the correct minimum degree threshold for $K_q$-decompositions quite a fascinating problem; at the current time, we would not hazard to guess whether the value in Theorem~\ref{thm:mainAsymptotic} is optimal or whether even better constructions may yet be found.

\subsection{Decompositions of \texorpdfstring{$q$}{q}-partite graphs}

Another one of the most well-studied objects in design theory is a \emph{Latin square} of order $n$, an $n$ by $n$ grid of cells each containing a symbol from $[n]$ such that every symbol appears exactly once in each row and column. From a graph theoretic perspective, this is equivalent to a triangle decomposition of the complete tripartite graph with $n$ vertices in each part, denoted $K_{n,n,n}$ (with the parts corresponding to rows, columns and symbols). The question of when a partial Latin square can be completed to a full Latin square has a rich history. Most relevant to our discourse here is the 1983 conjecture of Daykin and H{\"a}ggkvist~\cite{DH83} that a partial Latin square of order $n$ can be completed provided no row or column has more than $n/4$ entries and no symbol is used more than $n/4$ times.  

Daykin and H{\"a}ggkvist's Conjecture has a natural equivalent formulation from a graph theoretic perspective. For an integer $q\ge 2$, we say a graph $G$ is \emph{$q$-partite with $q$-partition $(V_1,\ldots,V_q)$} if $(V_1,\ldots,V_q)$ is a partition of $V(G)$ and for each $i\in [q]$, $V_i$ is an independent set in $G$; furthermore, if $|V_1|=|V_2|=\ldots=|V_q|$, we say such a graph is \emph{balanced}. For such a graph, we define its \emph{partite minimum degree} as
\[\hat{\delta}(G):=\min \{d(v, V_j): 1\leq j\leq q,\; v\in V(G)\setminus V_j\}.\]
In this graph theoretic setting, their conjecture is equivalent to conjecturing that if $G\subseteq K_{n,n,n}$ with $\hat{\delta}(G) \ge \frac{3}{4}n$ and $K_{n,n,n}\setminus G$ admits a triangle decomposition, then $G$ has a triangle decomposition. As in Nash-Williams' Conjecture, it seems unnecessary to assume the complement admits a triangle decomposition provided that $G$ satisfies certain divisibility constraints; here the necessary divisibility conditions for a balanced $q$-partite graph $G$ with $q$-partition $(V_1,\ldots,V_q)$ to admit a $K_q$-decomposition are that for all integers $1\leq j_1 < j_2\leq q$ and every $v\in V(G)\setminus(V_{j_1}\cup V_{j_2})$ we have $d(v,V_{j_1})=d(v,V_{j_2})$; we will call such a graph \emph{partite-$K_{q}$-divisible}. Here then is a partite version of Nash-Williams' Conjecture that generalizes Daykin and H{\"a}ggkvist's Conjecture (the conjecture is formally stated in Barber, K\"uhn, Lo, Osthus, and Taylor~\cite{BKLOT2017clique} but is implicit in Gustavsson~\cite{gustavsson1991decompositions}). 

\begin{conj}[Partite Version of Nash-Williams' Conjecture]\label{conj:PartiteNW}
The following holds for sufficiently large $n$: If $G$ is a balanced partite-$K_3$-divisible $3$-partite graph on $3n$ vertices with $\hat{\delta}(G)\ge \frac{3}{4}n$, then $G$ admits a $K_3$-decomposition.
\end{conj}

Once more, it is useful to consider the fractional relaxation of this conjecture; however, we note that in this partite setting, for $G$ to admit a fractional $K_3$-decomposition, it is necessary that $G$ is partite-$K_3$-divisible (since even `fractional' triangles have equal degree to the other parts). The best-known bound on the fractional relaxation is $0.96n$ due to Bowditch and Dukes~\cite{BD15} from 2015. In 2017 Barber, K\"uhn, Lo, Osthus, and Taylor~\cite{BKLOT2017clique} proved that if the fractional relaxation of Conjecture~\ref{conj:PartiteNW} holds for partite minimum degree $cn$ for some constant $c\ge 3/4$, then Conjecture~\ref{conj:PartiteNW} holds for any constant $c' > c$.

The generalization of the above questions to $K_q$-decompositions of $q$-partite graphs is quite interesting as $K_q$-decompositions of complete balanced $q$-partite graphs correspond to sets of mutually orthogonal Latin squares as follows. A pair of Latin squares of order $n$ is said to be \emph{orthogonal} if when superimposed, all ordered pairs of entries are distinct. From a graph theoretic perspective, an orthogonal pair of Latin squares of order $n$ is equivalent to a $K_4$-decomposition of the complete $4$-partite graph with parts of size $n$, denoted $K_{n,n,n,n}$ (with parts corresponding to the rows, columns, symbols from the first square, and symbols from the second square). More generally, a set of $q-2$ mutually orthogonal Latin squares of order $n$ is equivalent to a $K_q$-decomposition of the complete $q$-partite graph with parts of size $n$. 

The study of mutually orthogonal Latin squares (also known as transversal designs) has a rich history dating back centuries.  Especially of interest, it is rumored that in the 1700s Euler learned of a puzzle known as the ``Thirty-Six Officers Problem'' from Catherine the Great herself concerning the question of the existence of pairs of mutually orthogonal Latin squares of order 6~\cite{van2001course}. In 1782 Euler~\cite{E1782} proved the existence of pairs of mutually orthogonal Latin squares of order $n \equiv 0 \pmod 4$ and conjectured that there are no pairs of mutually orthogonal Latin squares of order $n \equiv 2 \pmod 4$.  In 1906 Tarry confirmed Euler's conjecture for $n=6$, settling the original ``Thirty-Six Officers Problem''. In 1959, Bose, Shrikhande, and Parker~\cite{BSP60, BS59} disproved Euler's Conjecture by showing that pairs exist for all such $n$ except 2 and 6; subsequently in 1960, using this work and incorporating additional insights from number theory, Chowla, Erd\H{o}s, and Straus \cite{CES60} showed that for each $q > 4$ and $n$ sufficiently large there exist $q-2$ mutually orthogonal Latin squares of order $n$.

The following partite version of the folklore conjecture, Conjecture~\ref{conj:Folklore}, was developed (it is formally stated by Barber, K\"uhn, Lo, Osthus, and Taylor~\cite{BKLOT2017clique} but is implicit in Gustavsson's work~\cite{gustavsson1991decompositions}).

\begin{conj}[Partite Version of the Folklore Conjecture]\label{conj:PartiteFolklore}
For each integer $q\ge 4$, the following holds for sufficiently large $n$: If $G$ is a balanced partite-$K_q$-divisible $q$-partite graph on $qn$ vertices with $\hat{\delta}(G)\ge \left(1-\frac{1}{q+1}\right)n$, then $G$ admits a $K_q$-decomposition.
\end{conj}

Again we note that in this partite setting for a $q$-partite graph $G$ to admit a fractional $K_q$-decomposition, it is necessary that $G$ is partite-$K_q$-divisible. The best-known bound on the fractional relaxation is $\hat{\delta}(G)\geq (1-1/10^6q^{3})n$ by Montgomery~\cite{M17Partite} in 2017. As for the best-known result on Conjecture~\ref{conj:PartiteFolklore}, Barber, K\"uhn, Lo, Osthus, and Taylor~\cite{BKLOT2017clique} in 2017 proved that if the fractional relaxation of Conjecture~\ref{conj:PartiteFolklore} holds for partite minimum degree $cn$ for some constant $c\ge 1-\frac{1}{q+1}$, then Conjecture~\ref{conj:PartiteFolklore} holds for any constant $c' > c$. 

We should note that Montgomery~\cite{M17Partite} specifically singled out the fractional relaxation of Conjecture~\ref{conj:PartiteFolklore} as interesting in its own right. That said, Montgomery~\cite{M17Partite} noted that the fractional relaxation of Conjecture~\ref{conj:PartiteFolklore} implies the fractional relaxation of Conjecture~\ref{conj:Folklore}, and indeed furthermore if the former held for partite minimum degree $cn$, then the latter would hold for minimum degree $cn$. Thus it follows from Theorem~\ref{thm:mainExistence} that we immediately disprove Conjecture~\ref{conj:PartiteFolklore} for all $q\ge 4$ as follows. We explain below the connection between these two conjectures for the interested reader.

\begin{thm}\label{thm:mainPartite}
For every integer $q\geq 4$, there exist $c>1$ and infinitely many balanced partite-$K_q$-divisible $q$-partite graphs $G$ with $\hat{\delta}(G)\geq \left(1-\frac{1}{c\cdot (q+1)}\right) \frac{v(G)}{q}$ with no $K_q$-decomposition and even no fractional $K_q$-decomposition. 
\end{thm}

Here is the connection observed by Montgomery~\cite{M17Partite}. Recall that the \emph{tensor product} of two graphs $G_1$ and $G_2$, denoted $G_1 \times G_2$, is the graph with $V(G_1\times G_2) = V(G_1) \times V(G_2)$ where two vertices $(v_1, v_2)$ and $(u_1, u_2)$ are adjacent if and only if $v_1u_1 \in E(G_1)$ and $v_2u_2 \in E(G_2)$. Notice that the graph $G\times K_q$ is a balanced $q$-partite graph with parts of size $v(G)$ and that $\hat{\delta}(G\times K_q) = \delta(G)$. Moreover $G\times K_q$ is partite-$K_q$-divisible for any $G$ by construction. Finally, we note that $G\times K_q$ admits a fractional $K_q$-decomposition if and only if $G$ does. (The backward direction follows by using for a copy of $K_q$ in $G\times K_q$ the weight of its original copy in $G$ scaled down by a factor of $(q-2)!$,  while the forward direction follows by using as the weight of a clique in $G$ the average of all its partite copies in $G\times K_q$.)

\subsection{Outline of Paper}

In Section~\ref{s:Overview}, we provide a proof overview. First in Section~\ref{ss:BlowUp}, we explain how it suffices to find one counterexample to the fractional relaxation for each $q$ and then blow up such a counterexample to yield infinitely many divisible counterexamples. Then in Section~\ref{ss:Intuition}, we overview the construction of a counterexample for each $q$ with a particular eye toward explaining why the cases $q=4$ and $q=5$ require their own special constructions.  In Section~\ref{sec:large}, we formally codify our general construction, use it to construct counterexamples for all $q\ge 6$ in Section~\ref{ss:AlmostAllq} and then finish the proof of Theorem~\ref{thm:mainAsymptotic} in Section~\ref{ss:LargeCase}. In Section~\ref{sec:smallcases}, we provide the construction for $q=5$ in Section~\ref{ss:q5} and for $q=4$ in Section~\ref{ss:q4}. 

\section{Proof Overview}\label{s:Overview}

\subsection{Blow-ups: Reducing to one counterexample for each \texorpdfstring{$q$}{q}}\label{ss:BlowUp}

In this subsection, we show how it suffices to find one counterexample to the fractional relaxation of Conjecture~\ref{conj:Folklore} for each $q$ and then blow up such a counterexample to yield infinitely many divisible counterexamples. To that end, we recall the definition of graph composition (also known as \emph{lexicographic product} or \emph{graph substitution}) as follows.  The \emph{graph composition} of two graphs $G_1$ and $G_2$, denoted $G_1 \circ G_2$, is the graph with $V(G_1\circ G_2) = V(G_1) \times V(G_2)$ where two vertices $(v_1, v_2)$ and $(u_1, u_2)$ are adjacent if and only if either $v_1u_1 \in E(G_1)$, or, $v_1=u_1$ and $v_2u_2 \in E(G_2)$. This composition can be thought of as `blowing up' each vertex of $G_1$ to a copy of $G_2$ (where each edge of $G_1$ then becomes a complete bipartite graph); we note that this operation is not necessarily commutative.

Given a counterexample $G$, we then consider the composition of $G$ with an independent set of $m$ vertices, i.e.,~$G\circ \overline{K_m}$ (sometimes called the \emph{$m$-blowup} of $G$); here are the salient properties for our purposes.

\begin{proposition}\label{prop:BlowingUp}
For each integer $m\ge 1$ and graph $G$, all of the following hold:
\begin{itemize}\itemsep.1in
    \item[(a)] $\frac{\delta(G\circ \overline{K_m})}{v(G\circ \overline{K_m})} = \frac{\delta(G)}{v(G)}$,
    \item[(b)] if $q(q-1)~|~m$, then $G\circ \overline{K_m}$ is $K_q$-divisible,
    \item[(c)] $G\circ \overline{K_m}$ admits a fractional $K_q$-decomposition if and only if $G$ does. 
\end{itemize}
\end{proposition}
\begin{proof}
Outcome (a) is immediate from the construction. Outcome (b) follows since the degree of each vertex of $G\circ \overline{K_m}$ is divisible by $m$, and hence by $q-1$, and in addition we also then find that $e(G\circ \overline{K_m})$ is divisible by $\frac{m}{2}$, and hence by $\binom{q}{2}$. For outcome (c), the backward direction follows by using for a copy of $K_q$ in $G\circ \overline{K_m}$ the weight of its original copy in $G$ scaled down by a factor of $m^{q-2}$, while the forward direction follows by using for the weight of a clique in $G$ the average of all its corresponding copies in $G\circ \overline{K_m}$.
\end{proof}

We note that outcome (c) above only holds for clique decompositions (as opposed to $F$-decompositions for non-complete $F$); this is because cliques in $G\circ \overline{K_m}$ correspond to cliques in $G$ while for a non-complete $F$, copies of $F$ in $G\circ \overline{K_m}$ may instead correspond to `identified' copies of $F$ in $G$ (for example consider that $K_2$ contains no $C_4$ but $K_2\circ \overline{K_m}$ has a fractional $C_4$-decomposition for all $m\ge 2$). 

Given Proposition~\ref{prop:BlowingUp}, in order to prove Theorems~\ref{thm:mainExistence} and~\ref{thm:mainAsymptotic} it suffices to prove the following theorems showing the existence of one counterexample $G$ (by taking $G\circ \overline{K_m}$ for all positive $m$ with $q(q-1)~|~m$). 

\begin{thm}\label{thm:mainExistence2}
     For each integer $q\geq 4$, there exist a $c>1$ and graph $G$ with minimum degree at least $\big(1-\frac{1}{c\cdot (q+1)}\big)v(G)$ with no fractional $K_q$-decomposition.
\end{thm}

\begin{thm}\label{thm:mainAsymptotic2}
    Let $c=\frac{1+\sqrt{2}}{2}$ and fix $\varepsilon>0$. For every large enough integer $q$, there exists a graph $G$ with minimum degree at least $\big(1-\frac{1}{(c-\varepsilon)\cdot (q+1)}\big)v(G)$ with no fractional $K_q$-decomposition.
\end{thm}

We prove Theorem~\ref{thm:mainExistence2} for $q\ge 6$ in Section~\ref{ss:AlmostAllq} and Theorem~\ref{thm:mainAsymptotic2} in Section~\ref{ss:LargeCase}. We prove Theorem~\ref{thm:mainExistence2} for $q=5$ in Section~\ref{ss:q5} and for $q=4$ in Section~\ref{ss:q4}.

\subsection{Intuition for the constructions}\label{ss:Intuition}

In this subsection, we discuss the construction of a specific counterexample for each $q$, providing intuition why this yields a counterexample, where the asymptotic value in Theorem~\ref{thm:mainAsymptotic2} comes from, and why the cases $q=4$ and $q=5$ each require their own special construction. The formal constructions and proofs of Theorems~\ref{thm:mainExistence2} and~\ref{thm:mainAsymptotic2} can be found later in the paper.

It may be useful to the reader to understand that there have been, in fact, two different constructions in the literature showing that the value of $1-\frac{1}{q+1}$ in Conjecture~\ref{conj:Folklore} would have been tight if true. The first construction due to Gustavsson in 1991~\cite{gustavsson1991decompositions} on pages 85-86 of his thesis (and reiterated by Yuster~\cite{Yuster2005asymptotically}), is $K_{q-1}\circ H$ where $H$ is a $d$-regular graph on $v$ vertices for some specific $d$ and $v$ with the property that $d/v$ is close to but smaller than $2/(q+1)$. The second construction, due to Garaschuk in 2014~\cite{garaschuk2014linear} (and reiterated by Barber, K\"uhn, Lo, and Osthus~\cite{BKLO16}), is $(K_{q+1}-M)\circ K_t$ (the $t$-clique-blowup) where $M$ is a perfect matching of $K_{q+1}$ (for $q$ odd). In both constructions, one can show that no fractional $K_q$-decomposition exists by arguing there are not enough `internal' edges to decompose the `cross' edges. 

Now let us endeavor to build a very general construction which, as will be seen, not only encapsulates both of the examples above but also provides our promised counterexamples.  With this construction in hand, we will determine for what parameters it does not admit a fractional $K_q$-decomposition and subsequently optimize said parameters in order to maximize the minimum degree.  To that end, we recall the following definition.

\begin{definition}[Join]\label{def:join}
The \emph{join} of graphs $H_1,\ldots,H_b$ is the graph obtained by taking the vertex-disjoint union of $H_1,\ldots,H_b$ and adding a complete bipartite graph between $V(H_i)$ and $V(H_j)$ for all $1\le i < j \le b$. 
\end{definition}

In our general construction we fix some $d$-regular graph $H$ on $v$ vertices and then consider $G:= K_b\circ H$, where $b$ is some integer to be specified (equivalently $G$ is the join of $b$ copies of $H$). We note the particular graph $H$ is not important to the construction but rather ultimately only the parameters $d$, $v$, and $b$ matter. We call an edge of $K_b\circ H$ between distinct copies of $H$ a \emph{cross edge} and an edge of $K_b\circ H$ that lies inside a copy of $H$ an \emph{internal edge}. Note there are exactly $b\cdot \frac{dv}{2}$ internal edges of $K_b\circ H$ and exactly $\binom{b}{2} \cdot v^2$ cross edges of $K_b\circ H$. The important observation is the following: 
\vskip.05in
\emph{if $b < q$, then every copy of $K_q$ in $K_b\circ H$ contains at least $q-b$ internal edges}.  
\vskip.05in
\noindent This observation follows from the pigeonhole principle (indeed by the convexity of binomial coefficients it follows that the minimum of internal edges is attained when the vertices of a copy of $K_q$ are distributed as equitably as possible. This observation implies that internal edges can only decompose (even fractionally) cross edges at a ratio of $(q-b)$ to $\binom{q}{2}-(q-b)$. However, if the ratio of internal edges to cross edges, which is equal to $\frac{ (bdv)/2} {\binom{b}{2} v^2} = \frac{d}{v} \cdot \frac{1}{b-1}$ is smaller than this ratio, then we see that $K_b\circ H$ does not admit a fractional $K_q$-decomposition, i.e.,~when 
\begin{equation}\label{eq:condition_dv}
\frac{d}{v} < (b-1)\cdot \frac{q-b}{\binom{q}{2}-(q-b)};    
\end{equation}

\noindent
see Lemma~\ref{lem:construction} for a formal proof.

Thus for any fixed integer $q$ and integer $b<q$, we can construct a graph with no fractional $K_q$-decompositions provided $H$ satisfies that $\frac{d}{v}$ is smaller than the number on the right side of the above inequality. This being said, \emph{what is the minimum degree of $K_b\circ H$?} Calculating, we find that 
$$\delta\left(K_b\circ H\right) = (b-1)v + d = bv-(v-d) = \left(1- \frac{1-\frac{d}{v}}{b}\right) \cdot bv= \left(1- \frac{1-\frac{d}{v}}{b}\right) v\left(K_b\circ H\right).$$
We note then that all that matters is the ratio $\frac{d}{v}$ for $H$. Furthermore, the minimum degree computed above increases as $\frac{d}{v}$ increases. Hence choosing $\frac{d}{v}$ arbitrarily close to the upper bound in the earlier inequality will yield graphs with no fractional $K_q$-decomposition with the largest minimum degree. What remains then is to maximize this minimum degree value over the range of $b$ smaller than $q$.

It is instructive now to return to the two known constructions mentioned earlier. First recall that Gustavsson's construction is of the form $K_{q-1}\circ H$. In our terms, this means that $b=q-1$ and hence the right side of (\ref{eq:condition_dv}) equals $(q-2)\cdot \frac{1}{\binom{q}{2}-1}= (q-2)\cdot \frac{2}{q^2-q-2} = \frac{2}{q+1}$. We note that this matches the degree ratio of $H$ from Gustavsson's construction and yields a minimum degree fraction of $1-\left(\frac{1}{q-1}\cdot\left(1-\frac{2}{q+1}\right)\right)=1-\left(\frac{1}{q-1}\cdot\left(\frac{q-1}{q+1}\right)\right) = 1 - \frac{1}{q+1}.$

Second, while at first glance Garaschuk's construction of $(K_{q+1}-M)\circ K_t$ may not appear to be of this form, we note her construction can be reinterpreted as $K_{(q+1)/2} \circ H$ where $H$ is the disjoint union of two copies of $K_t$ (so $H$ is a $(t-1)$-regular graph on $2t$ vertices). In our terms, this means that $b=\frac{q+1}{2}$ and $q-b=q-\frac{q+1}{2} = \frac{q-1}{2}$, and hence the right side of (\ref{eq:condition_dv}) equals 
$\left(\frac{q+1}{2} - 1\right)\cdot \frac{(q-1)/2}{\binom{q}{2}-(q-1)/2}= \left(\frac{q-1}{2}\right)\cdot \frac{(q-1)/2}{\binom{q}{2}-(q-1)/2} = \frac{(q-1)^2/4}{(q-1)^2/2} = \frac{1}{2}.$ We note that this matches the degree ratio of $H$ from our reinterpretation of Garaschuk's construction and yields a minimum degree fraction of $1- \frac{1 - \frac{1}{2}}{(q+1)/2} = 1 - \frac{1}{q+1}.$ 

It is useful to note that for $b=\frac{q}{2}$, the right side of~\cref{eq:condition_dv} gives $(\frac{q}{2} - 1)\cdot \frac{q/2}{\binom{q}{2}-q/2} = \frac{q(q-2)/4}{q(q-2)/2} = \frac{1}{2}$, and hence a minimum degree fraction of $1- \frac{1 - \frac{1}{2}}{q/2}=1-\frac{1}{q}$. So perhaps it is not surprising then to realize that the derived minimum degree value as a function of $b$ is a concave function from $b=\frac{q+1}{2}$ to $b=q-1$ with the two earlier constructions yielding matching values of $1-\frac{1}{q+1}$. We note that this means that \emph{any value of $b$ strictly between $\frac{q+1}{2}$ and $q-1$ provides counterexamples to Conjecture~\ref{conj:Folklore} once the ratio $\frac{d}{v}$ is optimized!} 

Now we arrive at the crux of why our general construction only works for $q\ge 6$: for every integer $q\ge 6$, there exists an \emph{integer} $b$ strictly between $\frac{q+1}{2}$ and $q-1$. However, for $q=5$, these values are $3$ and $4$ respectively; whereas for $q=4$, these values are $\frac{5}{2}$ and $3$. 

For these cases, our workaround is to add one more ``fractional part'' beyond the $\lfloor \frac{q+1}{2} \rfloor$ integral parts of the construction. For $q=5$, we target $b$ being slightly bigger than $3+\frac{4}{7} = \frac{25}{7}$; our construction then is obtained by the join of $3$ copies of $H$ and one copy of $I$, an independent set on slightly more than $\frac{4}{7}\cdot v$ vertices (and $\frac{d}{v}$ slightly less than $\frac{3}{7}$). This yields a minimum degree fraction approaching $1 - \frac{4/7}{25/7} = 1 - \frac{4}{25} = 1 - \frac{1}{6.25}$. Our contradiction, however, is more complicated in that we argue that the number of internal edges and edges to the fractional part are not enough to decompose the edges between copies of $H$; see Lemma~\ref{lem:SmallCaseFive} for the formal proof. For $q=4$, we target $b$ being around $2+(\sqrt{3}-1) \approx 2.71$ parts; our construction then is obtained by the join of $2$ copies of $H$ and one more graph $I$ on around $(\sqrt{3}-1)\cdot v$ vertices. However, here $I$ is too large to be an independent set and so we carefully choose it to be a regular graph of appropriate degree (around $0.196\cdot v$). This yields a minimum degree fraction of roughly $1-\frac{1}{5.098}$. Again the contradiction is more complicated; see Lemma~\ref{lem:SmallCaseFour} for the formal proof.

Notice that for $q=3$, the values $\frac{q+1}{2}$ and $q-1$ both equal $2$; so in this case there is no fractional value to target in order to obtain a larger minimum degree value. \emph{Indeed, it may well be that the behavior of triangles is special and the original conjecture of Nash-Williams could be true!}

As for the asymptotic estimate, we optimize as follows. Let $b = \beta\cdot q$ with $\beta \in [\frac{1}{2},1]$ to be optimized. The right side of our inequality~\eqref{eq:condition_dv} may be approximated for large $q$ as $b\cdot \frac{q-b}{q^2 /2 } =\beta\cdot q\cdot \frac{q-\beta\cdot q}{q^2 /2 } = 2\cdot \beta \cdot (1-\beta)$. This yields a minimum degree fraction of $1 - \frac{1 - 2\cdot \beta\cdot (1-\beta)}{\beta\cdot q}$. The right term may be simplified to $\frac{1-2\beta+2\beta^2}{\beta\cdot q} = \frac{1}{q} \cdot \left(\frac{1}{\beta} - 2 + 2\beta\right)$. Yet $\frac{1}{\beta} - 2 + 2\beta$ is a convex function in $\beta$ whose minimum is attained when the derivative, $-\frac{1}{\beta^2} + 2$, is set to zero. This yields $\beta = \frac{1}{\sqrt{2}}$ and a value for the function of $2\sqrt{2} - 2 \approx 0.828$. Note that $\frac{1}{2\sqrt{2}-2} = \frac{1+\sqrt{2}}{2} \approx 1.207$. Thus asymptotically, we arrive at a minimum degree fraction of $1 - \frac{2\sqrt{2}-2}{q} = 1 - \frac{1}{\frac{1+\sqrt{2}}{2}\cdot q}$. See the proof of Theorem~\ref{thm:mainAsymptotic2} in Section~\ref{ss:LargeCase} for a formal proof. 

We also note that the optimal value for $\frac{d}{v}$ for such a $b$ then is around $2\cdot \frac{1}{\sqrt{2}}\cdot \left(1-\frac{1}{\sqrt{2}}\right) = \sqrt{2}-1 \approx 0.414$. For example $C_5$ has degrees quite close to this ratio; indeed $K_5\circ C_5$ admits no fractional $K_7$-decomposition since it satisfies our inequality for $b=5$ and $q=7$ (namely $\frac{2}{5} < \frac{8}{19}$).  Yet $K_5\circ C_5$ has minimum degree $1-\frac{3}{25} = 1 - \frac{1}{8.3\overline{3}}$ and thus, is a counterexample to Conjecture~\ref{conj:Folklore} for $q=7$. 

\begin{remark}
The inquisitive reader may wonder what happens with our construction for $b < \frac{q}{2}$; these values of $b$ still provide counterexamples to Conjecture~\ref{conj:Folklore} but not as large a minimum degree value as obtained when $b\approx \frac{q}{\sqrt{2}}$. Suppose that $b=\frac{q}{k+\alpha}$ where $k\ge 1$ is integer and $\alpha \in [0,1]$. A stronger pigeonholing argument shows that much more than $q-b$ internal edges are used for $k\ge 2$; as mentioned before, the minimum value of internal edges is achieved when the vertices of a copy of $K_q$ are distributed as equitably as possible among the parts, that is, when using $k+1$ vertices from $\alpha\cdot b$ parts and $k$ vertices from $(1-\alpha)\cdot b$ parts. Hence at least $\alpha\cdot b\cdot \binom {k+1}{2} + (1-\alpha)\cdot b\cdot \binom{k}{2} = (k+2\alpha-1)\cdot \frac{kb}{2}$ internal edges are used. Call this number $\gamma$. It follows that $K_b\circ H$ admits no fractional $K_q$-decomposition if $\frac{d}{v} < (b-1) \cdot \frac{\gamma}{\binom{q}{2}-\gamma}$. For large $q$, this is approximately equal to 
\[b\cdot \frac{\gamma}{q^2/2} = (b^2/2)\cdot \frac{(k+2\alpha-1)\cdot k}{q^2/2} = \frac{(k+2\alpha-1)\cdot k}{(k+\alpha)^2}.\]
This in turn yields a minimum degree fraction of
\[1-\frac{1-\frac{(k+2\alpha-1)\cdot k}{(k+\alpha)^2}}{b} = 1 - \frac{(k+\alpha)^2 - (k+2\alpha-1)\cdot k}{(k+\alpha)\cdot q} = 1-\frac{k+\alpha^2}{(k+\alpha)\cdot q}.\] To maximize this, we minimize $\frac{k+\alpha^2}{k+\alpha} = 1 -\frac{\alpha \cdot (1-\alpha)}{k+\alpha}$ for $\alpha\in[0,1]$ and $k$ fixed. Clearly setting $k$ smaller for any fixed $\alpha$ will only make this value smaller (so the $k=1$ `interval' dominates all others); meanwhile for $k\ge 2$, the right term is at most $\frac{1/4}{2} \le \frac{1}{8}$ leading to a value no better than $\frac{8}{7}\cdot q \approx 1.14\cdot q$ in the denominator (which is smaller than the approximately $1.207q$ in Theorem~\ref{thm:mainAsymptotic2}).  Nevertheless for \emph{any} integer $k\ge 1$ and \emph{any} real $\alpha \in (0,1)$, the value $\frac{k+\alpha^2}{k+\alpha}$ is strictly less than $1$ and so would supply counterexamples to Conjecture~\ref{conj:Folklore} for $q$ large enough (even multiplicatively), see~\cref{fig:multiplicativeconstant} for the constant $c=c(k,\alpha)$ obtained. Viewed holistically from this asymptotic view then, all the ``integral'' ratios for number of parts (the \emph{harmonies}) yield a minimum degree value of $1-\frac{1}{q}$ asymptotically while the ``non-integral'' ratios (the \emph{discordances}) yield asymptotically larger minimum degree values with the largest such value lying in the first harmonic interval at a value of $\beta \approx \frac{1}{\sqrt{2}}$ (probably to the great dismay of the Pythagorean school).
\end{remark}

\begin{figure}[H]
    \centering
    \includegraphics[height=.3\linewidth,width=0.6\linewidth]{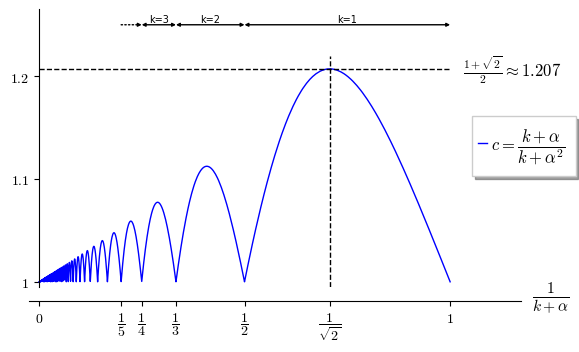}
    \caption{The constant $c$ in the minimum
degree fraction $1-\frac{1}{c\cdot q}$, for $b=\frac{q}{k+\alpha}$ with $k\in\mathbb{N}$ and $\alpha\in[0,1]$.}
    \label{fig:multiplicativeconstant}
\end{figure}

\section{The General Construction and Large \texorpdfstring{$q$}{q}}\label{sec:large}


The following lemma codifies when our general construction does not admit a fractional $K_q$-decomposition. First we need some additional notation as follows. For graphs $F,G$, we denote the set of copies of $F$ in $G$ by $\binom{G}{F}$. Given an assignment of weights $\phi$ to all all copies of $F$ in $G$, $\phi:\binom{G}{F}\to\mathbb{R}$, and given an edge $e$ of $G$, we let $\partial\phi(e)$ denote the total weight of $\phi$ over the edge $e$, that is, $\partial\phi(e):= \sum_{F: e\subseteq F} \phi(F)$.

\begin{lem}\label{lem:construction}
    For every integers $q,b,v,d\geq 1$ with $v\geq 2d$ and $2q>2b\geq q$, there exists a graph $G$ on $n$ vertices with $\delta(G)\geq \big(1-\frac{v-d}{bv}\big) n$ and such that, if 
    \[\frac{d}{v}<\frac{(b-1)\cdot (q-b)}{\binom{q}{2}-(q-b)}
    ,\] 
    then there exists no fractional $K_q$-decomposition of $G$.
\end{lem}

\begin{proof}
Let $H$ be a $d$-regular bipartite graph on $v$ vertices, and let $G := K_b\circ H$. Recall that we say that an edge inside a copy of $H$ is an {\em internal} edge , while all other edges are {\em cross} edges. Observe that $G$ contains $\frac{b\cdot d\cdot v}{2}$ internal edges, and $\binom{b}{2}v^2$ cross edges. Assume that there exists a fractional $K_q$-decomposition $\phi$ of $G$. By the pigeonhole principle, every copy of $K_q$ in $K_b\circ H$ contains at least $q-b$ internal edges. We obtain by counting the internal edges of $G$ that
\[\frac{b\cdot d\cdot v}{2} 
= \sum_{\substack{e\in G\\\textrm{internal }e}}\partial\phi(e)
= \sum_{\substack{e\in G\\\textrm{internal }e}} \sum_{H\colon e\in H} \phi(H) 
= \sum_{H\in \binom{G}{K_q}}\sum_{\substack{e\in H\\\textrm{internal }e}} \phi(H) 
\geq (q-b)\sum_{H\in\binom{G}{K_q}} \phi(H)\]
and similarly we obtain by counting for the cross edges of $G$ that
\[\binom{b}{2}v^2 
= \sum_{\substack{e\in G\\\textrm{cross }e}}\partial\phi(e)
= \sum_{\substack{e\in G\\\textrm{cross }e}} \sum_{H\colon e\in H} \phi(H) \leq \Big[\binom{q}{2}-(q-b)\Big]\sum_{H\in\binom{G}{K_q}} \phi(H).\]
Combining these two inequalities we obtain,
\[\frac{b\cdot d\cdot v}{2}\cdot\Big[\binom{q}{2}-(q-b)\Big]\geq \binom{b}{2}v^2 \cdot(q-b).\]
Hence
\[\frac{d}{v}\cdot \Big[\binom{q}{2}-(q-b)\Big]\geq (b-1)\cdot(q-b),\]
a contradiction with our assumption about the values of $q, b, v,$ and $d$. Therefore $G$ admits no fractional $K_q$-decomposition as desired.
\end{proof}

\subsection{A construction for all \texorpdfstring{$q\geq 6$}{q>=6}}\label{ss:AlmostAllq}

We are now ready to prove Theorem~\ref{thm:mainExistence2} when $q\geq 6$ as follows. 

\begin{lem}\label{lem:AlmostAllq}
    For every $q\geq 6$, there exist a $c>1$ and graph $G$ on $n$ vertices with 
    \[\delta(G)\geq \Big(1-\frac{1}{c\cdot(q+1)}\Big) n,\]
    such that there exists no fractional $K_q$-decomposition of $G$.
\end{lem}

\begin{proof}
    Fix an integer $q\geq 3$ and let $\varepsilon>0$ be an arbitrarily small rational compared to $1/q$. For a fixed integer $b<q$, we can define integers $d,v$ such that
    \[\frac{d}{v}=\frac{(b-1)\cdot (q-b)}{\binom{q}{2}-(q-b)+\varepsilon},\]
    and, by~\cref{lem:construction}, we obtain that there exists a graph $G$ on $n$ vertices with $\delta(G)\geq (1-\frac{v-d}{bv})\cdot n$ and admitting no fractional $K_q$-decomposition. Observe that $\frac{v-d}{bv}=\frac{1}{c\cdot(q+1)}$ with
    \[c := \frac{b\cdot\binom{q}{2}- b\cdot(q-b)+b\cdot \varepsilon}{\binom{q}{2}-b\cdot(q-b)+\varepsilon}\cdot\frac{1}{q+1}.\]
    Therefore, it remains to prove that there exists an integer $b$ with $2\leq b<q$ such that    
    \[\frac{b\cdot\binom{q}{2}- b\cdot(q-b)+b\cdot \varepsilon}{\binom{q}{2}-b\cdot(q-b)+\varepsilon}>(q+1).\]  
    
    \noindent As $\varepsilon$ is arbitrarily small compared to $1/q$, it suffices to find an integer $b$ with $2\leq b<q$ such that 
    \begin{equation}\label{eq:targetb}
        \frac{b\cdot\binom{q}{2}- b\cdot(q-b)}{\binom{q}{2}-b\cdot(q-b)}>(q+1).
    \end{equation}
    We then have
    \begin{align*}
        \cref{eq:targetb} 
        &\iff q^3-3\cdot b\cdot q^2 + q\cdot(2\cdot b^2+b-1)<0 \\
        &\iff 2\cdot b^2+ (1-3\cdot q)\cdot b +(q^2-1)<0\\
        &\iff \Big(b-(q-1)\Big)\cdot\Big(b-\frac{q+1}{2}\Big)<0.
    \end{align*}
\end{proof}

\begin{remark}
    The left-hand side of this inequality is a (convex) quadratic polynomial whose unique root is $b=q-1$ when $q=3$, but has two distinct roots for $q\geq 4$, namely \(b_+ := q-1\) and \(b_-:= \frac{q+1}{2}\).
    Observe that $|b_+-b_-|=1$ for $q=5$, and $|b_+-b_-|>1$ for $q\geq 6$. Therefore as discussed in the proof overview, for any $q\geq 6$, there exists an integer $b$ (e.g. $b=q-2$) satisfying~\cref{eq:targetb}, as desired.
\end{remark}

We note that, in particular, for every $q\geq 6$,~\cref{lem:AlmostAllq} holds with $c$ arbitrarily close to $\frac{(q-2)(q^2-q-4)}{(q^2-5q+8)(q+1)}$, obtained from~\cref{eq:targetb} with $b=q-2$.

\subsection{Large estimates}\label{ss:LargeCase}

We now prove Theorem~\ref{thm:mainAsymptotic2}, proving that $c$ can be arbitrarily close to $\frac{1+\sqrt{2}}{2}$ for large enough $q$.

\begin{proof}[Proof of~\cref{thm:mainAsymptotic2}]
    Let $d,v$ be integers such that 
    \[\sqrt{2}-1-\frac{\varepsilon}{4}\leq \frac{d}{v}\leq \sqrt{2}-1-\frac{\varepsilon}{5}.\] Let $q$ be a large enough integer compared to $\frac{1}{\varepsilon}$ and let $b:=\floor*{\frac{q+1}{\sqrt{2}}}$. 
    Then, as $q$ is large enough, we have 
    \[(b-1)(q-b) \geq 
    \Big(\frac{q+1}{\sqrt{2}}-2\Big)\Big(q - \frac{q+1}{\sqrt{2}}\Big)
    > \frac{q^2}{2}\Big(\sqrt{2}-1 - \frac{6-3\sqrt{2}}{q}\Big)
    \geq \frac{q^2}{2}\cdot\Big(\sqrt{2}-1-\frac{\varepsilon}{5}\Big).
    \]
    Therefore we have 
    \[\frac{d}{v}\leq\frac{2\cdot (b-1)\cdot (q-b)}{q^2}\leq \frac{(b-1)\cdot (q-b)}{\binom{q}{2}}\leq \frac{(b-1)\cdot (q-b)}{\binom{q}{2}-(q-b)}.\]  
    By~\cref{lem:construction}, there exists a graph $G$ on $n$ vertices with $\delta(G)\geq \big(1-\frac{v-d}{bv}\big)\cdot n$ such that $G$ has no fractional $K_q$-decomposition of $G$. Observe that
    \[\delta(G)\geq \Big(1-\frac{v-d}{bv}\Big)\cdot n = \Big(1-\frac{1}{Q}\Big)\cdot n, \]
    where 
    \[Q=\frac{bv}{v-d}=\Big(\frac{q+1}{\sqrt{2}}-1\Big)\cdot\frac{1}{1-\frac{d}{v}}\geq (q+1)\cdot \Big(\frac{1}{\sqrt{2}}\cdot\frac{1}{1-\frac{d}{v}}-\frac{\varepsilon}{4}\Big),\]
    where, in the last inequality, we used $q$ large enough compared to $\varepsilon$ and $d/v\leq\sqrt{2}-1$. 
    We then obtain
    \[Q
    \geq (q+1)\cdot \Big(\frac{1}{\sqrt{2}}\cdot\frac{1}{2-\sqrt{2}-\frac{\varepsilon}{4}}-\frac{\varepsilon}{4}\Big)
    \geq (q+1)\cdot \Big(\frac{1}{\sqrt{2}}\cdot\frac{1}{2-\sqrt{2}}-\varepsilon\Big)
    = (q+1)\cdot \Big(\frac{1+\sqrt{2}}{2}-\varepsilon\Big),\]
    hence 
    $\delta(G)\geq\left(1-\frac{1}{ \Big(\frac{1+\sqrt{2}}{2}-\varepsilon\Big)\cdot(q+1)}\right)n$, as desired.
\end{proof}

\section{The Small Cases}\label{sec:smallcases}

For $q \in \{4, 5\}$, the general construction does not provide a counterexample, so we turn to constructions specific to these values of $q$. 

\subsection{The case \texorpdfstring{$q=5$}{q=5}}\label{ss:q5}

We now prove Theorem~\ref{thm:mainExistence2} for $q=5$ as follows.

\begin{lem}\label{lem:SmallCaseFive}
    For $q = 5$, there exist a constant $c>1$ and graph $G$ such that $\delta(G) \geq \left(1-\frac{1}{c\cdot(q+1)}\right) v(G)$ and $G$ does not have a fractional $K_q$-decomposition. 
\end{lem}

\begin{proof}
    Let $\varepsilon>0$ be an arbitrarily small rational number, and let $\alpha := \frac{4}{7}+\varepsilon$. Let $v$ be an even integer such that $\alpha v\in \mathbb{N}$ (such $v$ exists as $\alpha$ is rational), and let $d:=(1-\alpha)<1/2$. Note $dv =v-\alpha v$ and hence $dv$ is integral. Let $H$ be a $(dv)$-regular balanced bipartite graph on $v$ vertices (such exists as $dv$ is integral and $dv \le v/2$), and let $I$ be an independent set of $\alpha v$ vertices. Let $G$ be the join of $H$, $H$, $H$ and $I$. 

    First observe that $G$ is $3v$-regular on $(3+\alpha)v$ vertices, and therefore we have 
    \[\delta(G)=\Big(1-\frac{\alpha}{3+\alpha}\Big)v(G)=\Big(1-\frac{1}{c\cdot (q+1)}\Big)v(G),\]
    with 
    \[c:=\frac{3+\alpha}{\alpha\cdot(q+1)}=\frac{25+7\cdot\varepsilon}{24+42\cdot\varepsilon},\]
    hence $c\geq 1.04$ for sufficiently small $\varepsilon$ as $\frac{25}{24} > 1.04$.
    Thus it suffices to prove that $G$ admits no fractional $K_q$-decomposition. 
    Assume for the purposes of contradiction that $G$ has a fractional $K_q$-decomposition $\phi$. It is helpful to consider the various types of edges of $G$. To that end, we say that an edge inside a copy of $H$ is an {\em internal} edge, that an edge between $I$ and $G\setminus I$ is a {\em small} edge, while all other edges are {\em cross} edges, see~\cref{fig:TikzCaseFive}.

    \begin{figure}[ht]
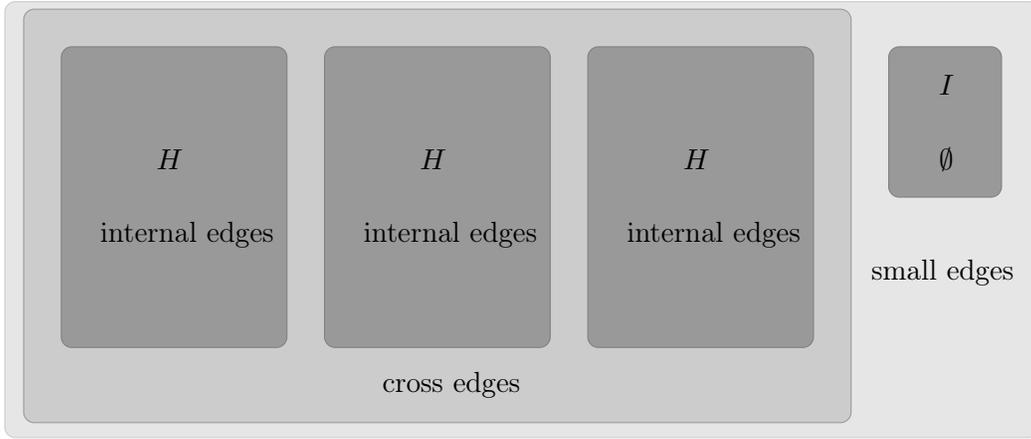

        \centering
        \TikzCaseFive{2}
        \caption{Internal, cross and small edges in the join of $H$, $H$, $H$, and $I$.}
        \label{fig:TikzCaseFive}
    \end{figure}

    Observe that $G$ contains all cross edges between copies of $H$, that is, $3v^2$ cross edges; each copy of $H$ is $(d v)$-regular on $v$ vertices, therefore $G$ contains $\frac{3 d v^2}{2}$ internal edges; finally $G$ contains all small edges between $I$ and every copy of $H$, that is, $3\alpha v^2$ small edges. Since $I$ is independent, we have that for any $F \in \binom{G}{K_q},$ $|V(F) \cap V(I)| \le 1$. For $i\in\{0,1\}$, let $\mc{F}_i$ be the family of all copies of $K_q$ in $G$ intersecting $I$ on exactly $i$ vertices, that is,
    \[\mc{F}_i := \set*{F\in \binom{G}{K_q}\colon |V(F)\cap I|=i},\]
    and let \[X_i:=\sum_{F\in \mc{F}_i}\phi(F).\]
    For every $F\in\mc{F}_0$, as $H$ is bipartite and therefore triangle-free, we note that $F$ contains $2$ internal edges and $8$ cross edges (and no small edges); on the other hand, for every $F\in\mc{F}_1$, we note that $F$ contains at least $1$ internal edge, at most $5$ cross edges, and exactly $4$ small edges. As $\phi$ is a fractional $K_q$-decomposition of $G$, we have that $\partial\phi(e)=1$ for every edge $e\in G$. In particular we have
    \[\big|\set{e\in G\colon e\textrm{ is an internal edge of } G}\big|=\sum_{\substack{e\in G\\\textrm{internal }e}} \partial\phi(e).\]
    Recall that $G$ contains $\frac{3 d v^2}{2}$ internal edges, therefore,  with $d=(1-\alpha)$, we obtain 
    \[\frac{3 (1-\alpha) v^2}{2} =
        \sum_{\substack{e\in G\\\textrm{internal }e}} \partial\phi(e) 
        =\sum_{\substack{e\in G\\\textrm{internal }e}} \sum_{F\colon e\in F} \phi(F) 
        = \sum_{F\in \binom{G}{K_q}}\sum_{\substack{e\in F\\\textrm{internal }e}} \phi(F).\]
    Then, recall that every $F\in\mc{F}_0$ contains $2$ internal edges, while every $F\in\mc{F}_1$ contains at least $1$ internal edge, therefore
    \begin{equation}
        \frac{3 (1-\alpha) v^2}{2} 
        = \sum_{F\in \binom{G}{K_q}}\sum_{\substack{e\in F\\\textrm{internal }e}} \phi(F)
        = \sum_{F\in \mc{F}_0}\sum_{\substack{e\in F\\\textrm{internal }e}} \phi(F)
        + \sum_{F\in \mc{F}_1}\sum_{\substack{e\in F\\\textrm{internal }e}} \phi(F)
        \geq 2X_0+X_1.\label{eq:internal}
    \end{equation} 
    Similarly for cross edges we obtain,
    \begin{alignat}{4}
        3v^2 &=&
        \sum_{\substack{e\in G\\\textrm{cross }e}} \partial\phi(e) 
        &=&\sum_{\substack{e\in G\\\textrm{cross }e}} \sum_{F\colon e\in F} \phi(F) 
        &=& \sum_{F\in \binom{G}{K_q}}\sum_{\substack{e\in F\\ \textrm{cross }e}} \phi(F)
        &\leq 8X_0+5X_1,\label{eq:cross}\\
        \intertext{and for small edges,}
        3\alpha v^2 &=&
        \sum_{\substack{e\in G\\\textrm{small }e}} \partial\phi(e) 
        &=&\sum_{\substack{e\in G\\\textrm{small }e}} \sum_{F\colon e\in F} \phi(F) 
        &=& \sum_{F\in \binom{G}{K_q}}\sum_{\substack{e\in F\\\textrm{small }e}} \phi(F)
        &= 4X_1.\label{eq:small}
    \end{alignat}
     By~\cref{eq:small}, we deduce $X_1 = \frac{3\alpha v^2}{4}$. Therefore, substituting this value of $X_1$ into~\cref{eq:internal,eq:cross}, we obtain $\frac{3(1-\alpha)v^2}{2} \geq  2X_0+\frac{3\alpha v^2}{4}$ and $3v^2\leq 8X_0+\frac{15\alpha v^2}{4}$, respectively.
    Isolating $X_0$ we deduce that,
    \[\frac{1}{8}\Big(3-\frac{15\alpha}{4}\Big)v^2\leq X_0\leq \frac{1}{2}\Big(\frac{3(1-\alpha)}{2}-\frac{3\alpha}{4}\Big)v^2.\]
    Simplifying by $\frac{v^2}{4}>0$, it follows that
    \[\frac18(12-15\alpha)\leq 3(1-\alpha)-\frac{3\alpha}{2},\]
    and therefore $7\alpha\leq 4$, which is the desired contradiction since $\alpha > \frac{4}{7}$ by definition. 
\end{proof}

\subsection{The case \texorpdfstring{$q=4$}{q=4}}\label{ss:q4}

The construction for $q = 4$ is similar to that of $q=5$: we take the join of two larger dense graphs $H$ and one small sparse graph $I$. Yet, in this case, taking $I$ to be independent does not provide a counterexample, but rather leads to an extremal example for Conjecture \ref{conj:Folklore}. Instead, we allow $I$ to be $d$-regular for a small value $d$. We do not include our optimization of our two parameters in the following argument, but we follow that same principles as in the proof of Lemma \ref{lem:SmallCaseFive}.

We now prove Theorem~\ref{thm:mainExistence2} for $q=4$ as follows.

\begin{lem}\label{lem:SmallCaseFour}
    For $q = 4$, there exist a constant $c>1$ and graph $G$ such that $\delta(G) \geq \left(1-\frac{1}{c(q+1)}\right)v(G)$ and $G$ does not have a fractional $K_q$-decomposition. 
\end{lem}

\begin{proof}
    Let $\alpha\in(\frac12,1)$ be a rational number. Let $d_1:=\frac{4\alpha-2}{4+\alpha}-\varepsilon\leq \frac{\alpha}{2}$ for some arbitrarily small rational number $\varepsilon$, and let $d_2 := d_1 +(1-\alpha)\leq \frac{1}{2}$. 
    
    Let $v$ be an even integer such that $\alpha \cdot v$, $d_2\cdot v$, and therefore $d_1\cdot  v$ are even integers. Let $H$ be a $(d_2\cdot v)$-regular balanced bipartite graph on $v$ vertices, and let $I$ be an $(d_1\cdot  v)$-regular balanced bipartite graph on $\alpha\cdot  v$ vertices. Let $G$ be the join of $H$, $H$ and $I$. 
    
    Observe that, by definition of $d_2$, we have that $G$ is $[(2+d_1)\cdot v]$-regular on $n:=(2+\alpha)\cdot v$ vertices. Thus,
    \[\delta(G)=n-(\alpha-d_1)\cdot v 
    = \Big(1-\frac{\alpha-d_1}{2+\alpha}\Big)\cdot n
    =\Big(1-\frac{1}{c\cdot (q+1)}\Big)\cdot n,\]
    with 
    \[c:=\frac{2+\alpha}{5\cdot (\alpha-d_1)}=\frac{(2+\alpha)(4+\alpha)}{5\cdot (\alpha^2+2+\varepsilon)}.\]
    By taking $\varepsilon>0$ arbitrarily small and $\alpha$ arbitrarily close to $\sqrt{3}-1$ we obtain that $c$ is arbitrarily close to $\frac{5+3\sqrt{3}}{10}\approx1.0196$.
    
    It remains to show that $G$ does not admit a fractional decomposition. For a contradiction, assume that $G$ contains a fractional $K_q$-decomposition $\phi$. We now consider the various types of edges in $G$. We say that an edge inside a copy of $H$ is an {\em $H$-internal} edge, an edge inside $I$ is an {\em $I$-internal} edge, that an edge between $I$ and $G\setminus I$ is a {\em small} edge, while edges between copies of $H$ are {\em cross} edges, see~\cref{fig:TikzCaseFour}.

   \begin{figure}[ht]
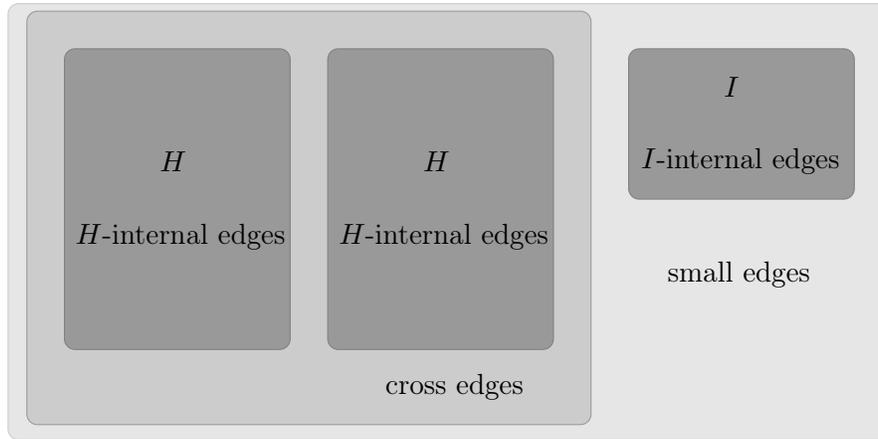

        \centering
        \TikzCaseFour{2}
        \caption{$H$-internal, cross, small and $I$-internal edges in the join of $H$, $H$, and $I$.}
        \label{fig:TikzCaseFour}
    \end{figure}

    Observe that $G$ contains all cross edges between copies of $H$, that is, $v^2$ cross edges; each copy of $H$ is $(d_2\cdot v)$-regular on $v$ vertices, therefore $G$ contains $d_2\cdot v^2$ $H$-internal edges; $I$ is $(d_1\cdot v)$-regular on $\alpha\cdot v$ vertices, therefore $G$ contains $\frac{\alpha\cdot d_1\cdot v^2}{2}$ $I$-internal edges; finally $G$ contains all small edges between $I$ and both copies of $H$, that is, $2\alpha v^2$ small edges. Since $I$ is a bipartite graph, we have that for any $F \in \binom{G}{K_q},$ $|V(F) \cap V(I)| \le 2$. For $i\in\{0,1,2\}$, let $\mc{F}_i$ be the family of all copies of $K_q$ intersecting $I$ in exactly $i$ vertices, that is,
    \[\mc{F}_i = \set*{F\in \binom{G}{K_q}\colon |V(F)\cap I|=i},\]
    and let 
    \[X_i = \sum_{F\in\mc{F}_i}\phi(F).\]
     Recall that $H$ is bipartite and so, for every $F \in \binom{G}{K_q},$ we have $|V(F) \cap V(H)| \le 2$. Thus we note that for every $F\in \mc{F}_0$, $F$ contains $2$ $H$-internal edges and $4$ cross edges. For every $F\in \mc{F}_1$, $F$ contains $1$ $H$-internal edge and $2$ cross edges (and $3$ small edges). For every $F\in \mc{F}_2$, $F$ contains $1$ $I$-internal edges, at most $1$ cross edge, and at most $1$ $H$-internal edge (and $4$ small edges).
    Therefore, similarly to the proof of~\cref{lem:SmallCaseFive}, we obtain, respectively from cross, $H$-internal and $I$-internal edges,
    \begin{alignat}{3}
            4X_0+2X_1 \leq &~v^2   &&\leq 4X_0+2X_1+X_2,&\label{eq:CaseFourcross}\\
        2X_0+X_1  \leq &~d_2 \cdot v^2 &&\leq 2X_0+X_1+ X_2,&\label{eq:CaseFourH}\\
        &\dfrac{\alpha \cdot d_1 \cdot v^2}{2}&&=X_2.\label{eq:CaseFourI}&    
    \end{alignat}
    Combining~\cref{eq:CaseFourcross,eq:CaseFourH} we deduce that 
    \(v^2-X_2\leq 4x_0+2X_1\leq 2\cdot d_2\cdot v^2,\)
    hence using~\cref{eq:CaseFourI} we obtain that
    \[v^2\leq 2 \cdot d_2 \cdot v^2 + \frac{\alpha\cdot d_1 \cdot v^2}{2}.\]
    By the definition of $d_2$ we get $1\leq 2\cdot (d_1+(1-\alpha)) + \frac{\alpha \cdot d_1 }{2},$
    and hence
    $d_1\geq\frac{4\alpha-2}{4+\alpha},$ contradicting the definition of $d_1$. Therefore $G$ has no fractional $K_q$-decomposition, as desired.
\end{proof}

\section{Acknowledgements}
The authors would like to thank both referees for several helpful comments.  The first author would like to thank Victor Falgas Ravry and Maryam Sharifzadeh for assistance in locating the original thesis of Gustavsson.

\bibliographystyle{plain}
\bibliography{bibliography}

\end{document}